\documentclass[a4paper,12pt]{amsart}
\usepackage[cp1250]{inputenc}
\usepackage[T1]{fontenc}
\usepackage[hmargin={3.5cm,3.5cm}]{geometry}
\usepackage{amsthm}
\usepackage[leqno]{amsmath}
\usepackage{amssymb}
\usepackage{amsfonts}
\usepackage{enumerate}

\newcommand{\rr}{\mathbf{\mathbb{R}}}

\newcommand{\nn}{\mathbf{\mathbb{N}}}

\newcommand{\vf}{\varphi}

\newcommand{\dist}{\operatorname{dist}}

\newcommand{\Go}{\operatorname{\varrho_0}}
\newtheorem{lemma}{Lemma}
\newtheorem{theorem}{Theorem}

\newtheorem*{nonmath}{Main Theorem}
\newtheorem{st}{Step}

\theoremstyle{remark}

\title{$C^r$-right equivalence of analytic functions}

\author{Piotr Migus}
\thanks{This research was partially supported by the Polish OPUS Grant No 2012/07/B/ST1/03293.}
\address{Faculty of Mathematics and Computer Science, Wydzia\l{} Matematyki i Informatyki, Uniwersytet \L{}\'o{}dzki, Banacha 22, 90-238 \L{}\'o{}d\'z{}, Poland}
\date{03.12.2014}
\email{migus@math.uni.lodz.pl}

\begin{document}
\begin{abstract}
Let $f,g:(\rr^n,0)\rightarrow (\rr,0)$ be analytic functions. We will show that if $\nabla f(0)=0$ and $g-f \in (f)^{r+2}$ then $f$ and $g$ are $C^r$-right equivalent, where $(f)$ denote ideal generated by $f$ and $r\in \nn$. 
\end{abstract}

\maketitle

\section{Introduction and result}
By $\nn$ we denote the set of positive integers. A norm in $\rr^n$ we denote by $|\cdot|$ and by $\dist(x,V)$ - the distance of a point $x\in \rr^n$ to a set $V\subset\rr^n$ (put $\dist(x,V)=1$ if $V=\emptyset$).

Let $f,g:(\rr^n,0)\rightarrow (\rr,0)$ be analytic functions. We say that $f$ and $g$ are $C^r$-\emph{right} \emph{equivalent} if there exists a $C^r$ diffeomorphism $\vf:(\rr^n,0)\rightarrow (\rr^n,0)$ such that $f=g \circ \vf$ in a neighbourhood of $0$. 

Let $f:(\rr^n,0)\rightarrow \rr$ be an analytic function. By $\mathcal{J}_f$ we denote the ideal generated by $\frac{\partial f}{\partial x_1},\dots ,\frac{\partial f}{\partial x_n}$ in the set of analytic functions $(\rr^n,0)\rightarrow \rr$. The ideal $\mathcal{J}_f$ is called the \emph{Jacobi ideal}. Moreover, by $(f)$ we denote the ideal in set of analytic functions $(\rr^n,0)\rightarrow \rr$ generated by $f$.

The aim of this paper is proof of the following theorem

\begin{nonmath}\label{the:1}
Let $f,g:(\rr^n,0)\rightarrow (\rr,0)$ be analytic functions and let $\nabla f(0)=0$. If $(g-f)\in (f)^{r+2}$ then $f$ and $g$ are $C^r$-right equivalent, where $r\in \nn$.
\end{nonmath}

The above theorem is a modification of author's result about $C^r$-right equivalence of $C^{r+1}$ functions. In \cite[Theorem 5]{mrs} and \cite[Theorem 1]{mig} it has been proved 

\begin{theorem}\label{the:2}
Let $f,g:(\rr^n,0)\rightarrow (\rr,0)$ be $C^k$ functions, $k,r\in \nn$ be such that $k\geq r+1$ and let $\nabla f(0)=0$. If $(g-f)\in (\mathcal{J}_fC^{k-1}(n))^{r+2}$ then $f$ and $g$ are $C^r$-right equivalent. By $\mathcal{J}_fC^{k-1}(n)$ we mean the Jacobi ideal defined in the set of $C^{k-1}$ functions $(\rr^n,0)\rightarrow \rr$. 
\end{theorem}

Methods of proofs of above theorems are similar. First we construct suitable vector field of class $C^r$ and next we integrate this vector field. The idea of construct vector field is descended from N.~H.~Kuiper, T.~C.~Kuo (\cite{Kui}, \cite{Kuo}). Whereas, integration of vector field is descended from  Ch.~Ehresmann (\cite{ehr}, see also \cite{gg}).

There exists one more result which deals with $C^r$-right equivalence of functions with similar condition for $(g-f)$. Namely, J.~Bochnak has proved the following theorem (\cite[Theorem 1]{boch})

\begin{theorem}\label{the:3}
Let $f,g:(\rr^n,0)\rightarrow (\rr,0)$ be $C^k$ functions, $k,r\in \nn$ be such that $k\geq r+2$ and let $\nabla f(0)=0$. If $(g-f)\in \mathfrak{m}(\mathcal{J}_fC^{k-1}(n))^{2}$ then $f$ and $g$ are $C^r$-right equivalent. By $\mathcal{J}_fC^{k-1}(n)$ and $\mathfrak{m}$ we mean respectively the Jacobi ideal and maximal ideal defined in the set of $C^{k-1}$ functions $(\rr^n,0)\rightarrow \rr$. 
\end{theorem}

Proof of this theorem bases on Tougeron's Implicit Theorem (\cite{tou}).

Comparing the above results we see that Theorem \ref{the:2} deals with $C^r$-right equivalence of $C^{r+1}$ functions, whereas Theorem \ref{the:3} deals with  $C^r$-right equivalence of $C^{r+2}$ functions. Since in the last Theorem power of Jacobi ideal does not depend on $r$, so it is difficult to say which Theorem is stronger. Additional, since in Main Theorem $(g-f)$ belongs to some power of ideal generated by $f$, whereas in Theorem \ref{the:2} and Theorem \ref{the:3} $(g-f)$ belongs to some power of ideal generated by partial derivatives of $f$, so this results are completely other type.

\section{Auxiliary results}\label{sec:2}

We start from define {\L}ojasiewicz exponent in the gradient inequality. 

Let $f:(\rr^n,0)\rightarrow(\rr,0)$ be an analytic function. It is known that there exists a neighbourhood $U$ of $0\in \rr^n$ and constants $C>0$, $\eta \in [0,1)$ such that the following \emph{{\L}ojasiewicz gradient inequality} holds
$$
|\nabla f(x)|\geq C|f(x)|^{\eta}, \quad \textrm{ for }x\in U.
$$
The smallest exponent $\eta$ in the above inequality is called the \emph{{\L}ojasiewicz exponent in the gradient inequality} and is denoted by $\Go(f)$ (cf. \cite{Loj2}, \cite{Loj3}).

From the above inequality we obtain immediately that there exists a neighbourhood $U$ of $0\in\rr^n$ and a constant $C>0$ such that
\begin{equation}\label{eq:lo}
|\nabla f(x)|\geq C|f(x)|, \quad \textrm{ for }x\in U.
\end{equation}

\vspace{0.5cm}

Let $M,m,r\in\nn$, $M>r$. Moreover, let $p,q_1,\dots,q_m:(\rr^n,0)\rightarrow \rr$ be analytic functions and let $\mathcal{Q}$ denote the ideal generated by $q_1,\dots,q_m$. 

\begin{lemma}[see \cite{mig}]\label{lem:1}
If $p\in \mathcal{Q}^M$ then 
\begin{itemize}
\item[(i)] $\frac{\partial^r p}{\partial x_{i_1} \dots \partial x_{i_r}}\in \mathcal{Q}^{M-r}$ for $i_1,\dots,i_r \in \{1,\dots,n\}$,
\item[(ii)] $|p(x)|\leq  C |(q_1(x),\dots,q_n(x))|^M$ in a neighbourhood of $0$ and for some positive constant $C$.
\end{itemize}
\end{lemma}

\begin{lemma}\label{lem:2}
Let $f:(\rr^n,0)\rightarrow(\rr,0)$ be an analytic function. Then there exist a neighbourhood $U$ at $0\in \rr^n$, constant $C>0$ such that for any $x\in U$, $\left| f(x)\right| \leq C \dist (x,V_f) $ ($V_f$ denote zero set of $f$).
\end{lemma}

\begin{proof} 
Let us assume contrary, that for any neighbourhood $U$ and for any $C>0$ there exists $x \in U$, $\left| f(x)\right| > C \dist(x,V_f)$. In particular for any $\nu \in \mathbb{N}$ there exists $x_{\nu}$, such that $\left| x_{\nu}\right|<\frac{1}{\nu}$, $\left| f(x_{\nu})\right| > \nu \dist(x_{\nu},V_f)$. Moreover there exists $u_{\nu} \in V_f$, that $\dist(x_{\nu},V_f)=\left| x_{\nu}-u_{\nu} \right|$. Then we have $\left| f(x_{\nu})-f(u_{\nu})\right| >$ $\nu \left|x_{\nu}-u_{\nu}\right|$. This contradicts the Lipschitz condition for function $f$. 
\end{proof}

\begin{lemma}\label{lem.tech}
Let $\xi,\eta:U\rightarrow \rr$ be $C^k$ functions such that
\begin{equation}\label{lem.tech0}
A_1|\eta(x)|^2 \leq|\xi(x)|\leq A_2|\eta(x)|^2, \quad |\partial \xi(x)|\leq A_3 |\eta(x)|, \quad x\in U,
\end{equation}
where $A_1,A_2,A_3>0$ are some positive constants and $U\in \rr^n$ is some neighbourhood of the origin,. Then
\begin{equation}\label{lem.tech1}
\left|\partial^k \left( \frac{1}{\xi(x)}\right)\right|\leq B|\eta(x)|^{-|k|-2},\quad x \in U,
\end{equation}
for some constant $B>0$. $k\in\nn_0^n$.
\end{lemma}

\begin{proof}
Let $m=|k|$. By induction  it is easy to show that 
\begin{equation}\label{lem.tech2}
\partial^k \left( \frac{1}{\xi}\right) = \frac{1}{\xi^{m+1}} \left(\sum_{j=1}^{m}\xi^{m-j}\sum_{|i_1|+\dots +|i_j|=m}C_{i_1,\dots,i_j}\partial^{i_1}\xi\cdots\partial^{i_j}\xi\right)
\end{equation}
where $i_1,\dots, i_j\in \nn_0^n$, $i_1+\dots + i_j =k$, $|i_j|\geq 1$ and for some constants $C_{i_1,\dots, i_j}\geq 0$ ($C_{i_1,\dots, i_j}= 0$, when $i_1+\dots + i_j \neq k$). 

Now we will prove \eqref{lem.tech1}. Let us take $k\in \nn_0^n$ and let $|k|=m$. First, consider the case when $m$ is even. 
\begin{align*}
&\left|\frac{1}{\xi^{m+1}} \left(\sum_{j=1}^{m}\xi^{m-j}\sum_{|i_1|+\dots +|i_j|=m}C_{i_1,\dots,i_j}\partial^{i_1}\xi\cdots\partial^{i_j}\xi\right) \right|\\
& \leq \left| \frac{1}{\xi^{m+1}} \left(\sum_{j=1}^{\frac{1}{2}m}\xi^{m-j}\sum_{|i_1|+\dots +|i_j|=m}C_{i_1,\dots,i_j}\partial^{i_1}\xi\cdots\partial^{i_j}\xi\right)\right|\\
& + \left| \frac{1}{\xi^{m+1}} \left(\sum_{j=\frac{1}{2}m+1}^{m}\xi^{m-j}\sum_{|i_1|+\dots +|i_j|=m}C_{i_1,\dots,i_j}\partial^{i_1}\xi\cdots\partial^{i_j}\xi\right)\right|\\
\end{align*}
Note that for $m\geq j\geq \frac{1}{2}m+1$ and for any sequence $i_1, \dots, i_j\in \nn_0^n$, $|i_{j}|\geq 1$, such that $|i_1|+\dots +|i_j|=m$, there exist at least $2j-m$ elements of this sequence which modules are equal $1$. Therefore we can assume that $|i_{m-j+1}|=\dots |i_j|=1$ for $m\geq j\geq \frac{1}{2}m+1$. From this and \eqref{lem.tech0} we obtain
\begin{align*}
&\left|\frac{1}{\xi^{m+1}} \left(\sum_{j=1}^{m}\xi^{m-j}\sum_{|i_1|+\dots +|i_j|=m}C_{i_1,\dots,i_j}\partial^{i_1}\xi\cdots\partial^{i_j}\xi\right) \right|\\
&\leq A_1|\eta|^{-2m-2} \left| \sum_{j=1}^{\frac{1}{2}m}\xi^{m-\frac{1}{2}m}\sum_{|i_1|+\dots +|i_j|=m}C_{i_1,\dots,i_j}\partial^{i_1}\xi\cdots\partial^{i_j}\xi\right|\\
&+ A_1|\eta|^{-2m-2}\left| \sum_{j=\frac{1}{2}m+1}^{m}\xi^{m-j}\sum_{|i_1|+\dots +|i_j|=m}C_{i_1,\dots,i_j}\partial^{i_1}\xi\cdots\partial^{i_{m-j}}\xi \partial^{i_{m-j+1}}\xi\cdots \partial^{i_j}\xi\right|
\end{align*}
\begin{align*}
& \leq A_1|\eta|^{-2m-2} A_2 B_1|\eta|^{2(m-\frac{1}{2}m)}+ A_1|\eta|^{-2m-2} \left| \sum_{j=\frac{1}{2}m+1}^{m}B_2\xi^{m-j}\partial^{i_{m-j+1}}\xi\cdots \partial^{i_j}\xi\right|\\
& \leq A_1 A_2 B_1 |\eta|^{-m-2} + A_1 |\eta|^{-2m-2}\sum_{j=\frac{1}{2}m+1}^{m}A_2A_3B_2|\eta |^{2(m-j)+2j-m}\\
& = B_3|\eta|^{-m-2},
\end{align*}
where $B_1,B_2,B_3$ are some positive constants.

Let us consider the case when $m$ is odd. Note that for $m\geq j\geq \frac{1}{2}(m+1)$ and for any sequaence $i_1, \dots, i_j\in \nn_0^n$, $|i_{j}|\geq 1$, such that $|i_1|+\dots +|i_j|=m$, there exist at least $2j-m$ elements of this sequence which modules are equal $1$. Knowing this fact, similar as previously, we show
\begin{align*}
&\left|\frac{1}{\xi^{m+1}} \left(\sum_{j=1}^{m}\xi^{m-j}\sum_{|i_1|+\dots +|i_j|=m}C_{i_1,\dots,i_j}\partial^{i_1}\xi\cdots\partial^{i_j}\xi\right) \right|\\
& \leq \left| \frac{1}{\xi^{m+1}} \left(\sum_{j=1}^{\frac{1}{2}(m-1)}\xi^{m-j}\sum_{|i_1|+\dots +|i_j|=m}C_{i_1,\dots,i_j}\partial^{i_1}\xi\cdots\partial^{i_j}\xi\right)\right|\\
& + \left| \frac{1}{\xi^{m+1}} \left(\sum_{j=\frac{1}{2}(m+1)}^{m}\xi^{m-j}\sum_{|i_1|+\dots +|i_j|=m}C_{i_1,\dots,i_j}\partial^{i_1}\xi\cdots\partial^{i_j}\xi\right)\right|\\
& \leq B_4 |\eta|^{-2m-2}|\eta|^{2(m-\frac{1}{2}m+\frac{1}{2})} + B_5|\eta|^{-m-2}\\
\end{align*}
for some positive constants $B_4,B_5$. Finally, we proved \eqref{lem.tech1}.
\end{proof}

\section{Proof of Main Theorem}
Let $Z$ be the zero set of $\nabla f$ and let $U\in \rr^n$ be a neighbourhood of $0$ such that $f$ and $g$ are well defined. By Lemma \ref{lem:2} there exists a positive constant $A$ such that
\begin{equation}\label{E:5}
|\nabla f(x)|\leq A \dist(x,Z)\quad\textrm{ for }x\in U.
\end{equation}
Define the function $F:\rr^{n}\times U\rightarrow \rr$ by the formula
$$
F(\xi,x)=f(x)+\xi(g-f)(x),
$$
obviously
$$
\nabla F(\xi,x)=\left((g-f)(x),\nabla f(x)+\xi\nabla (g-f)(x)\right).
$$
Let $G=\{(\xi,x)\in\rr\times U:|\xi|< \delta\}$ where $\delta \in \nn$, $\delta>2$. From the above, diminishing $U$ if necessary, we have that there exists a constant $C_1>0$ such that
\begin{equation}\label{E:7}
|\nabla f(x)|\leq C_1|\nabla F(\xi,x)|\quad\textrm{ for }(\xi,x)\in G. 
\end{equation}
Indeed,
$$
|\nabla F(\xi,x)|\geq |\nabla f(x)-\xi \nabla (g-f)(x)|\geq |\nabla f(x)|- |\xi||\nabla (g-f)(x)|.
$$
Since $(g-f)\in (f)^{r+2}$ and $r\geq 1$, so from Lemma \ref{lem:1} and \eqref{eq:lo} we get 
$$
|\nabla (g-f)(x)|\leq C'_2|f(x)|^{r+1}\leq C_2|\nabla f(x)|^{r+1} \leq C_2|\nabla f(x)|^2
$$
for some positive constants $C_2,C'_2$. Hence, diminishing $U$ if necessary, 
$$
|\nabla F(\xi,x)|\geq |\nabla f(x)|- |\xi|C_2|\nabla f(x)|^2\geq C_1|\nabla f(x)| \quad \textrm{ for } (\xi,x)\in G.
$$ 
Moreover, from definition of $\nabla F$ we get at once, that there exists a positive constant $C_3$ such that
\begin{equation}\label{E:7'}
|\nabla f(x)|\geq C_3|\nabla F(\xi,x)|\quad\textrm{ for }(\xi,x)\in G.
\end{equation}

Now we will show that the mapping $X:G\rightarrow \mathbb{R}^n\times \mathbb{R}$ defined by
$$
X(\xi,x)=(X_1,\dots,X_{n+1})= \left\{ \begin{array}{ll}
\frac{(g-f)(x)}{\left|\nabla F(\xi,x) \right|^2}\nabla F(\xi,x) & \textrm{ for $x\notin Z$}\\
0 & \textrm{ for $x\in Z$}
\end{array} \right.
$$
is a $C^r$ mapping. The proof of this fact will be divided into several steps

\begin{st}\label{st:1}
The mapping $X$ is continuous in $G$.
\end{st}
Indeed, let us fix $\xi$ and let $h_i(\xi,x)=(g-f)(x)\frac{\partial F}{\partial x_i}(\xi,x)$. Then for $x\in U$ and $x \notin Z$, from \eqref{eq:lo} and Lemma \ref{lem:1} we have $|X_i(\xi,x)|\leq A_1|\nabla f(x)|^{r+1}\leq A' \dist(x,Z)^{r+1}$
for some positive constants $A_1,A'$. The above inequality also holds for $x\in Z$. Since $A'$ does not depend on the choice of $\xi$ so for $(\xi,x)\in G$ we obtain
\begin{equation}\label{E:8}
|X(\xi,x)|\leq A' \dist(x,Z)^{r+1}.
\end{equation}
Therefore $X$ is continuous in $G$.

\begin{st}\label{st:2}
Let $\alpha=(\alpha_0, \dots , \alpha_n) \in \nn_0^{n+1}$ be a multi-index such that $|\alpha|\leq r$, then, diminishing $U$ if necessary, 
$$
|\partial^{\alpha} X_i (\xi,x)|\leq A''\dist(x,Z)^{r+1-|\alpha|} \textrm{ for } x\notin Z.
$$ 
where $\partial ^ {\alpha} X_i=\partial ^{\alpha_0}\cdots \partial^{\alpha_{n+1}}X_i=\frac{\partial^{|\alpha|}X_i}{\partial \xi^{\alpha_0}\partial x_1^{\alpha_1} \cdots \partial x_n^{\alpha_n}}.$
\end{st}
Indeed, from Leibniz rule we have
\begin{equation}\label{E:9}
\partial^{\alpha} X_i(\xi,x)=\sum_{\beta \leq \alpha} \binom{\alpha}{\beta}\partial^{\alpha-\beta}(h_i(\xi,x))\partial^{\beta}\left( \frac{1}{|\nabla F(\xi,x)|^2}\right).
\end{equation}

Diminishing $G$ if necessary, from Lemma \ref{lem.tech} we obtain
$$
\left|\partial^{\beta}\left( \frac{1}{|\nabla F(\xi,x)|^2}\right)\right|\leq\frac{A''_{\beta}}{|\nabla F(\xi,x)|^{|\beta|+2}},
$$
for some constants $A''_{\beta}>0$. Therefore from \eqref{E:9} we have
\begin{equation}\label{E:15}
|\partial^{\alpha}X_i(\xi,x)|\leq \sum_{\beta\leq \alpha}\binom{\alpha}{\beta}|\partial^{\alpha-\beta}(h_i(\xi,x))|\frac{A''_{\beta}}{|\nabla F(\xi,x)|^{|\beta|+2}}.
\end{equation}
Let us fix $\xi$. From Lemma \ref{lem:1}, \eqref{E:7'} and \eqref{eq:lo}  we have 
\begin{equation}\label{E:14}
|\partial^{\alpha-\beta}(h_i(\xi,x))|\leq B_{\alpha-\beta} |\nabla f(x)|^{r+3-|\alpha|+|\beta|}
\end{equation}
for some positive constant $B_{\alpha-\beta}$. Since $B_{\alpha-\beta}$ doesn't depend on the choice of $\xi$ so this equality holds for $(\xi,x)\in G$. Finally from \eqref{E:15}, \eqref{E:14} \eqref{E:7}, \eqref{E:7'} and \eqref{E:5} we obtain
\begin{align*}
|\partial^{\alpha} X_i(\xi,x)|&\leq \sum_{\beta \leq \alpha}\binom{\alpha}{\beta}B_{\alpha-\beta}|\nabla f(x)|^{r+3-|\alpha|+|\beta|}\frac{A''_{\beta}}{|\nabla F(\xi,x)|^{|\beta|+2}}\\
&\leq \sum_{\beta \leq \alpha}\binom{\alpha}{\beta}A''_{\beta}B_{\alpha-\beta}|\nabla f(x)|^{r+3-|\alpha|+|\beta|-|\beta|-2}\\
&\leq \frac{A''}{A} |\nabla f(x)|^{r+1-|\alpha|}\leq A'' \dist(x,Z)^{r+1-|\alpha|},
\end{align*}
for some constant $A''>0$.

\begin{st}\label{st:3}
Partial derivatives $\partial^{\alpha} X_i$ vanish for $x\in Z$ and $|\alpha|\leq r$.
\end{st}

Indeed, we will  carry out induction with respect to $|\alpha|$. Let $t\in \rr$, $x \in Z$ and let $x^t_m=(x_1,\dots,x_m+t,\dots, x_n)$. For $|\alpha|=0$ hypothesis is obvious. Assume that hypothesis is true for $|\alpha|\leq r-1$. Then from Step \ref{st:2} we have
\begin{align*}
\frac{|\partial^{\alpha}X_i(\xi,x^t_m)-\partial^{\alpha}X_i(\xi,x)|}{|t|}&=\frac{|\partial^{\alpha}X_i(\xi,x^t_m)|}{|t|}\leq \frac{A''\dist(x^t_m,Z)^{r+1-|\alpha|}}{|t|}\\
&\leq \frac{A''|t|^{r+1-|\alpha|}}{|t|}=A''|t|^{r-|\alpha|}.
\end{align*}
Since $r-|\alpha|\geq r-r+1= 1$, we obtain $\partial^{\gamma}X_i(\xi,X)=0$ for $x\in Z$ and $|\gamma|=|\alpha|+1$. This completes Step \ref{st:3}.

In summary from Step \ref{st:1}, \ref{st:2} and \ref{st:3} we obtain that $X_i$ are $C^r$ functions in $G$. Therefore $X$ is a $C^r$ mapping in $G$.

Define a vector field $W:G\rightarrow \rr^n$ by the formula
$$
W(\xi,x)=\frac{1}{X_1(\xi,x)-1}(X_2(\xi,x),\dots,X_{n+1}(\xi,x)).
$$
Diminishing $U$ if necessary, we may assume that $A'\dist(x,Z)< \frac{1}{2}$. From \eqref{E:8} we obtain
$$
\left|X_1(\xi,x)-1\right|\geq 1-\left|X(\xi,x)\right|\geq 1-A'\dist(x,Z)>\frac{1}{2}, \quad (\xi,x)\in G.
$$
Hence the field $W$ is well defined and it is a $C^r$ mapping. 

Consider the following system of ordinary differential equations
\begin{equation}\label{E:12}
\frac{dy}{dt}=W(t,y). 
\end{equation}
Since $r\geq 1$, then $W$ is at least of class ${C}^1$ on $G$, so it is a lipschitzian vector field. As a consequence, the above system has a uniqueness of solutions property in $G$. Since $y_0(t)=0$, $t \in (-2,2)$ is one of the solutions of \eqref{E:12}, then the above implies the existence of a neigbourhood $U\subset \mathbb{R}^n$ of $0$ such that every integral solution $y_x$ of \eqref{E:12}  with $y_x(0)=x$, where $x\in U$, is defined at least in $[0,1]$.

Now, let us define a mapping $\varphi:U\rightarrow \mathbb{R}^n$ by the formula
$$
\varphi(x)=y_x(1),
$$  
where $y_x$ stands for an integral solution of \eqref{E:12} with $y_x(0)=x$. This mapping is $\varphi$ is a $C^r$ bijection.It gives a $C^r$ diffeomorphism of some neighbourhood of the origin. Indeed, considering solutions $\bar{y}_x:[0,1]\rightarrow \mathbb{R}^n$ of \eqref{E:12} with $\bar{y}_x(1)=x$, where $x$ is from some neigbourhood of the origin, we get that $\varphi(\bar{y}_x(0))=x$. Similar reasoning shows that the mapping $x \rightarrow \bar{y}_x(0)$ for $x$ is class $C^r$ in the neigbourhood of the origin. Consequently $\varphi :(\mathbb{R}^n,0)\rightarrow (\mathbb{R}^n,0)$ is a $C^r$ diffeomorphism and maps a neighbourhood of the origin onto a neighbourhood of the origin. 

Finally, note that for any $x\in U$, 
\begin{equation}\label{E:13}
F(t,y_x(t))=const. \quad \textrm{ in }[0,1].
\end{equation}
Indeed, from definition of $W$ we derive the formula
$$
[1,W(\xi,x)]=\frac{1}{X_1(\xi,x)-1}(X(\xi,x)-e_1)\quad \textrm{ for }(\xi,x)\in G,
$$
where $e_1=[1,0,\dots,0]\in\mathbb{R}^{n+1}$ and $[1,W]:G\rightarrow \mathbb{R}\times \mathbb{R}^n$. Thus, if we denote by $\left\langle a,b\right\rangle$ the scalar product of two vectors $a$, $b$, then for $t\in [0,1]$, we have
\begin{align*}
&\frac{dF(t,y_x(t))}{dt}= \left\langle (\nabla F)(t,y_x(t)),[1,W(t,y_x(t))]\right\rangle\\
&=\frac{1}{X_1(t,y_x(t))-1}\left( \left\langle (\nabla_x F)(t,y_x(t)),X(t,y_x(t))\right\rangle -\frac{\partial F}{\partial \xi}(t,y_x(t))\right)\\
&=\frac{1}{X_1(t,y_x(t))-1}(g(y_x(t))-f(y_x(t))-g(y_x(t))+f(y_x(t)))=0.
\end{align*}
This gives \eqref{E:13}. Finally, \eqref{E:13} yields
$$
f(x)=F(0,x)=F(0,y_x(0))=F(1,y_x(1))=F(1,\varphi(x))=g(\varphi(x)).
$$
for $x\in U$. This ends the proof.
 
$\hfill \Box$

\section{Remark}
 
In Main Theorem we can not omit the assumption about analtyticity of function $f$ and $g$. It follows from the fact that the {\L}ojasiewicz gradient inequality holds only for analytic functions.

Note that the condition $g-f \in (f)^{r+2}$ in Main Theorem can be replaced by $g=f(hf^{r+1}+1)$, where $h:(\rr^n,0)\rightarrow \rr$ is an analytic function. It seems natural to try to replace this condition by $g=hf$, where $h:(\rr^n,0)\rightarrow \rr$ is an analytic function such that $h(0)\neq 0$. But then the theorem would not hold. Indeed, let $f(x)=x^2$, $g(x)=-x^2$ and $h(x)=-1$, then $g=hf$ but $f$ and $g$ are not right equivalent.

\end{document}